\documentclass{amsart}
 
\usepackage{amsmath, amsthm, amscd, amsfonts, amssymb, graphicx, color} 
\usepackage[bookmarksnumbered, colorlinks, plainpages]{hyperref}
\usepackage{blkarray}
\usepackage{arydshln,mathtools}
 \newtheorem{thm}{Theorem}[section]
 \newtheorem{cor}[thm]{Corollary}
 \newtheorem{lem}[thm]{Lemma}
 
 \theoremstyle{definition}
 
 \theoremstyle{problem} 
  
 \newtheorem{exam}[thm]{Example}
 \theoremstyle{remark}
 
 \numberwithin{equation}{section}

\title[The energy and spectrum of non-commuting graphs]
 {The energy and spectrum of non-commuting graphs}
\author {M.Nasiri $^1$ and S.H.Jafari $^2$}
\begin{document}
\email {nasirimaryam.63@gmail.com\\
shjafari55@gmail.com}
 \address {Department of Mathematics, Shahrood University of Technology, Shahrood, Iran, P.O. Box: 316-3619995161.}


\begin{abstract} 
Let $G$ be a non-abelian group and $Z(G)$ be the center of $G$. The non-commuting graph $\Gamma(G)$ of $G$ is a graph with vertex set $G-Z(G)$ in which two vertices $x$ and $y$ are joined if and only if $xy\neq yx$. In this paper we calculate the energy, Laplacian energy and spectrum of non-commuting graph of dihedral group $D_{2n}$. Also we will obtain the energy of non-commuting graph of $D_{2n}\times D_{2n}$ and $G\times H$, where G is a non-abelian finite group and H is an abelian finite group. 
\end{abstract}


\maketitle

\section{Introduction and preliminaries}
Let $G$ be a non-abelian finite group and $Z(G)$ be its center. The non-commuting graph $\Gamma(G)$ of $G$ is a graph whose vertex set is $G-Z(G)$ and  two vertices $x$ and $y$ are joined if and only if $xy\neq yx$. Note that if $G$ is an abelian, then $\Gamma (G)$ is the null graph.The non-commuting graph $\Gamma (G)$ was first considered by Paul Erd\"os, when he posed the following problem in 1975 $[7]$: Let $G$ be a group whose non-commuting graph has no infinite complete subgraph. Is it true that there is a finite bound on the cardinalities of complete subgraphs of $\Gamma (G)$? B. H. Neumann answered positively Erd\"os's question in $[7]$. 
The adjacency matrix of graph $\Gamma$ is the $(0,1)$ matrix $A$ indexed by the vertex set $V(\Gamma)$ of $\Gamma$, where $A_{xy}=1$ when there is an edge from $x$ to $y$ in $\Gamma$ and $A_{xy}=0$ otherwise.
The characteristic polynomial of $A$, denoted by $P_{A}(x)$, is the polynomial defined by $P_{A}(x)=det(xI-A)$ where $I$ denotes the identity matrix. 
The spectrum of a finite graph $\Gamma $ is by definition the spectrum of the adjacency matrix $A$, that is, its set of eigenvalues together with their multiplicities. Assume that $\lambda_{1}\geq\lambda_{2}\geq \ldots \geq \lambda_{t}$ are $t$ distinct eigenvalues of $\Gamma$ with the corresponding multiplicities $k_{1},k_{2},\ldots,k_{t}$. We denote by 
$$spec(\Gamma)=\begin{pmatrix}
\lambda_{1} & \lambda_{2} & \ldots & \lambda_{t} \\
k_{1} & k_{2} & \ldots & k_{t} 
\end{pmatrix}.
$$
Let $\Gamma$ be an undirected graph without loops. The Laplacian matrix of $\Gamma$ is the matrix $L$ indexed by the vertex set of $\Gamma$, with zero row sums. If $D$ is the diagonal matrix, indexed by the vertex set of $\Gamma$ such that $D_{xx}$ is the degree of $x$ then $L=D-A$. The energy of a graph $\Gamma$, denoted by $E(\Gamma)$, is defined as 
 $$E(\Gamma)=\sum^{n}_{i=1}\left|\lambda_i\right|,$$
where $\lambda_1\geq\lambda_2\geq\ldots\geq\lambda_n$ are the eigenvalues of the adjacency matrix of $\Gamma$. This concept was introduced by Gutman and is intensively studied in chemistry, since it can be used to approximate the total $\pi$-electron energy of a molecule(see, e.g.[3,4]).\\
Let $\Gamma$ be a graph with $n$ vertices and $m$ edges. Let $\mu_1,\mu_2,\ldots,\mu_n$ be the Laplacian eigenvalues of $\Gamma$. The Laplacian energy of a graph $\Gamma$, is defined as
$$LE(G)=\sum_{i=1}^{n}{\left| \mu_{i}-\frac{2m}{n}\right|}.$$
\begin{thm}\label{1} 
([2,6]). Let $K_{n_{1},n_{2},\ldots,n_{p}}$ denote the complete $p$-partite graph, \\
$p\geq 1$, $n=n_{1}+n_{2}+\ldots+n_{p}$ and $n_{1}\geq n_{2}\geq\ldots\geq n_{p}> 0$. Then
$$P(K_{n_{1},n_{2},\ldots,n_{p}},\lambda)=\lambda ^{(n-p)}(1-\sum _{i=1}^{p}\frac{n_{i}}{\lambda+n_{i}})\prod_{j=1}^{p}{(\lambda+n_{j})}.$$
\end{thm} 
The spectrum of $K_{n_{1},n_{2},\ldots,n_{p}}$ consist of the spectral radius $\lambda_{1}$ determined from the equation $\sum_{i=1}^{p}\frac{n_{i}}{\lambda+n_{i}}=1$, eigenvalue $0$ with multiplicity $n-p$ and $p-1$ eigenvalues situated in the intervals $[-n_{p},-n_{p-1}],\ldots,[-n_{2},-n_{1}]$.\\
\begin{lem}
([5])
If $\lambda_1$ is the spectral redius of the complete multipartite graph $K_{n_{1},n_{2},\ldots,n_{p}}$, then \\
$$E( K_{n_{1},n_{2},\ldots,n_{p}})=2\lambda_{1}.$$
\end{lem}
\begin{lem}\label{2}
([5]). $(1)$ If $p=2$, then
$$P(K_{n_{1},n_{2}},\lambda)=\lambda^{n-2}(\lambda^{2}-n_{1}n_{2}),$$
$$\lambda_{1}(K{n_{1},n_{2}})=\sqrt{n_{1}n_{2}},\hspace{2cm} E(K_{n_{1},n_{2}})=2\sqrt{n_{1}n_{2}}.$$
$(2)$ If $p=3$, then 
$$P (K_{n_{1},n_{2},n_{3}},\lambda)=\lambda^{(n-3)}(\lambda^{3}-(n_{1}n_{2}+n_{2}n_{3}+n_{3}n_{1})\lambda -2n_{1}n_{2}n_{3}).$$
\end{lem}
\section{The energy and Laplacian energy of non-commuting graphs of some special groups }
In this section, we calculate the energy and Laplacian energy of non-commuting graph of dihedral group $D_{2n}$.
\begin{thm}\label{4}
\item[(1)]
If $n$ is even and $n>4$, then 
$$
spec(\Gamma_{D_{2n}})=\begin{pmatrix}
-2 & 0 &\frac{ (n-2)- \sqrt{5n^{2}-12n+4}}{2} & \frac{(n-2)+ \sqrt{5n^{2}-12n+4}}{2} \\
& & & \\
\frac{n}{2} -1& \frac{3n}{2}-3 & 1 & 1
\end{pmatrix}.
$$
\item[(2)]
If $n$ is odd, then 
$$
spec(\Gamma_{D_{2n}})=\begin{pmatrix}
-1 & 0 &\frac{ (n-1)- \sqrt{5n^{2}-6n+1}}{2} & \frac{(n-1)+ \sqrt{5n^{2}-6n+1}}{2} \\
& & & \\
(n-1)& (n-2) & 1 & 1
\end{pmatrix}.
$$
\item[(3)]
If $n=4$, then 
$spec(\Gamma_{D_{2n}})=\begin{pmatrix}
-2 & 0 &4 \\
2&3& 1
\end{pmatrix}
$.
\end{thm}
\begin{proof}
The adjacency matrix of $\Gamma_{D_{2n}}$ is equal to 
$$A_{ij}(\Gamma _{D_{2n}})=\left\{ 
\begin{array}{rl}
0 & 1\leq i, j\leq n-2 \\
0 & i=k+t , j=k+s ; t, s=0 \ or \ 1 \\
& and \ k=n-1, n+1, \ldots , 2n-3 \\
1 & o.w
\end{array}\right.
$$
when n is even, and 
$$A_{ij}(\Gamma _{D_{2n}})=\left\{ 
\begin{array}{rl}
0 & 1\leq i, j\leq n-1 \\
0 & i=j=n, \ n+1, \ \ldots, \ 2n-1\\
1 & o.w
\end{array}\right.
$$
when $n$ is odd. By direct calculations
$$
P_{\Gamma_{D_{2n}}}(x) =\left\{
\begin{array}{rl}
(-x)^{\frac{3n}{2} -3}(-x-2)^{\frac{n}{2}-1} (x^{2} -x(n-2)-n(n-2)) & \text{if} \ \ n \ is \ even \\
(-x)^{ n-2}(-x-1)^{n-1} (x^{2} -x(n-1)-n(n-1)) & \text{if} \ \ n \ is \ odd.
\end{array}\right.
$$
This completes the proof.
\end{proof}
\begin{cor}
$$
E(\Gamma_{D_{2n}})=\left\{
\begin{array}{rl}
{(n-2)}+\sqrt{5n^2-12n+4} & \text{if}  \ \ n \ is \ even\\
(n-1)+\sqrt{5n^{2}-6n+1} & \text{if}  \ \  n \ is \ odd.
\end{array}\right.
$$
\end{cor}
In Table $1$, the energies of some non-commuting graphs of dihedral groups is given.\\ 

\begin {table}[h] 
\small
\begin{center} 
\begin {tabular}{|l|l|l|l|}
\hline
groups & characteristic polynomial & eigenvalues & energy\\
$D_{6}$ & $P_{A}(x)=(-x)(-x-1)^{2}(x^{2}-2x-6)$ & $(0)_{1},(-1)_{2},(1+\sqrt{7}),(1-\sqrt{7})$ & $2+2\sqrt{7}$ \\
$D_{8}$ & $P_{A}(x)=(-x)^{3}(-x+4)(-x-2)^{2}$ & $(0)_{3},(4)_{1},(-2)_{2}$ & $8$\\
$D_{10}$ & $P_{A}(x)=(-x)^{3}(-x-1)^{4}(x^{2}-4x-20)$ & $(0)_{3},(-1)_{4},(2-2\sqrt{6}), (2+2\sqrt{6})$ & $4+4\sqrt{6}$\\
$D_{12}$ & $P_{A}(x)=(-x)^{6}(-x-2)^{2}(x^{2}-4x-24)$& $(0)_{6},(-2)_{2},(2+2\sqrt{7}),(2-2\sqrt{7})$ & $4+4\sqrt{7}$ \\ 
$D_{16}$ & $P_{A}(x)=(-x)^{9}(-x-2)^{3}(x^{2}-6x-48)$ & $(0)_{9},(-2)_{3},(3+\sqrt{57}),(3-\sqrt{57})$ & $6+2\sqrt{57}$ \\ 
\hline
\end{tabular} 
\end{center} 
\caption{The characteristic polynomial , eigenvalues and energy of some graphs.} 
\end{table} 

\begin{thm} \label{99}
Let $G=GL{(2,q)}$, where $q=p^n > 2$  (p is prime). Then
\begin{align*}
P_{\Gamma_{G}}(x)&=x^{(n-t)}[x^3+(-q^{4}+q^{3}+4q^{2}-6q+2)x^{2}\\
&\qquad +(-2q^{6}+6q^{5}-q^{4}-13q^{3}+15q^{2}-5q)x-(q-1)^{4}q^{2}(q-2)(q+1)]\\
&\qquad (x+(q-1)^{2})^{q}(x+q(q-1))^{\frac{q^{2}-q-2}{2}}(x+(q-1)(q-2))^{\frac{q^{2}+q-2}{2}},
\end{align*}
where $t=q^{2}+q+1$.
\end{thm}
\begin{proof}
By [1] we have \\
$$G-Z{(G)}=(\dot{\bigcup}_{g\in G}{((PZ(G))^{g}-Z(G))}\cup{(\dot{\bigcup}_{g\in G}I^{g}-Z(G))}\cup{(\dot{\bigcup}_{g\in G}D^{g}-Z(G))},$$
such that  $D$ is the subgroup of all diagonal matrices in $G$, $I$ is a cyclic subgroup of $G$ of order $q^{2}-1$ and $P$ is the sylow $p$-subgroug of $G$ contaning all matrices as $\begin{bmatrix} 1 & x \\ 0 & 1 \end{bmatrix}$. Moreover  $\left|{D}\right|={(q-1)}^2$ and $\left|{PZ{(G)}}\right| =q(q-1)$. Also the number of conjugates of $D$, $I$ and $PZ(G)$ is equal to $\frac{q(q+1)}{2}$, $\frac{q(q-1)}{2}$ and $q+1$, respectively. Since $C_{G}(d)=D$ for any non-central element $d$ of $D$, $C_{G}(I)=I$ and $PZ(G)=C_{G}(x)$ for any non-trivial element of $P$, then the non-commuting graph of group $G$ is a complete $t$-partite graph where 
$$t=\frac{q(q+1)}{2}+\frac{q(q-1)}{2}+q+1=q^{2}+q+1.$$
By Theorem \ref{1},\\ 
\begin{align*}
P_{\Gamma_{G}}(x) &=x^{(n-t)}\Big[x^3+(-q^{4}+q^{3}+4q^{2}-6q+2)x^{2}\\
&\qquad +(-2q^{6}+6q^{5}-q^{4}-13q^{3}+15q^{2}-5q)x-(q-1)^{4}q^{2}(q-2)(q+1)\Big]\\
&\qquad (x+(q-1)^{2})^{q}(x+q(q-1))^{\frac{q^{2}-q-2}{2}}(x+(q-1)(q-2))^{\frac{q^{2}+q-2}{2}}.
\end{align*}
\end{proof}
\begin{cor}
$$
E(\Gamma_{GL(2,q)})=
|\gamma_{1}|+|\gamma_{2}|+|\gamma_{3}|+ q(q-1)^{2}+\frac{q^{2}-q-2}{2}q(q-1)
+ \frac{q^{2}+q-2}{2}(q-1)(q-2),
$$
where $\gamma_{1}$, $\gamma_{2}$ and $\gamma_{3}$ are roots of $f(x)=x^{3}+(-q^{4}+q^{3}+4q^{2}-6q+2)x^{2}+(-2q^{6}+6q^{5}-q^{4}-13q^{3}+15q^{2}-5q)x-(q-1)^{4}q^{2}(q-2)(q+1)$.
\end{cor}
\begin{proof}
By Theorem \ref{99}, we have 
\begin{align*}
P_{\Gamma_{G}}(x) &=x^{(n-t)}[x^3+(-q^{4}+q^{3}+4q^{2}-6q+2)x^{2}\\
&\qquad +(-2q^{6}+6q^{5}-q^{4}-13q^{3}+15q^{2}-5q)x-(q-1)^{4}q^{2}(q-2)(q+1)]\\
&\qquad (x+(q-1)^{2})^{q}(x+q(q-1))^{\frac{q^{2}-q-2}{2}}(x+(q-1)(q-2))^{\frac{q^{2}+q-2}{2}}.
\end{align*}
Let $f(x)=x^{3}+(-q^{4}+q^{3}+4q^{2}-6q+2)x^{2}+(-2q^{6}+6q^{5}-q^{4}-13q^{3}+15q^{2}-5q)x-(q-1)^{4}q^{2}(q-2)(q+1)$, $b=(-q^{4}+q^{3}+4q^{2}-6q+2)$, \\
$c=(-2q^{6}+6q^{5}-q^{4}-13q^{3}+15q^{2}-5q)$ and 
$d=-(q-1)^{4}q^{2}(q-2)(q+1)$. Then we have $f(x)=x^{3}+bx^{2}+cx+d$. It is convenient to make the translation $x=y-\frac{b}{3}$, converting $f(x)$ into $g(y)=f(y-\frac{b}{3})=y^{3}+\alpha y+\beta$, where 
$\alpha =\frac{b^{2}-2b}{3}+c$ and $\beta=\frac{-b^{3}+3b^{2}-9bc}{27}+d$.
We have
$$\Delta =(\frac{\alpha}{3})^{3}+(\frac{\beta}{2})^{2}.$$
Since $\Delta <0$, $g(y)$ has three real roots. Now let $\gamma_{1}$, $\gamma_{2}$ and $\gamma_{3}$ be roots of $g(y)$. Then
\begin{align*}
E(\Gamma_{GL(2,q)})&=\sum_{i=1}^{n}{\left|\lambda_{i}\right|}\\
&=|\gamma_{1}|+|\gamma_{2}|+|\gamma_{3}|+ q\left| -(q-1)^{2}\right| +\frac{q^{2}-q-2}{2}\left| -q(q-1)\right| \\
&+ \frac{q^{2}+q-2}{2}\left| -(q-1)(q-2)\right|\\
&=|\gamma_{1}|+|\gamma_{2}|+|\gamma_{3}|+ q(q-1)^{2}+\frac{q^{2}-q-2}{2}q(q-1)\\
&+ \frac{q^{2}+q-2}{2}(q-1)(q-2).\\
\end{align*}
\end{proof}
\begin{thm} 
\item[(1)]
If $n$ is even and $n>4$, then 
$$
spec(L(\Gamma _{D_{2n}}))=\begin{pmatrix}
2n-2 & 2n-4 & n & 0 \\
\frac{n}{2} & \frac{n}{2} & n-3 & 1
\end{pmatrix}.
$$
\item[(2)]
If $n$ is odd, then
$$
spec(L(\Gamma _{D_{2n}}))=\begin{pmatrix}
0 & n & 2n-1  \\
1 & n-2 & n
\end{pmatrix}.
$$
\item[(3)]
If $n=4$, then 
$spec(L(\Gamma_{D_{2n}}))=\begin{pmatrix}
0 & 4 & 6 \\
1&3& 2
\end{pmatrix}
$.
\end{thm}
\begin{proof}
By considering $L(\Gamma _{D_{2n}})$ and direct calculations, we have 
$$
P_{L(\Gamma _{D_{2n}})}{(x)}=\left\{
\begin{array}{rl}
x{(x-n)}^{(n-3)}{(x-{(2n-4)})}^{(\frac{n}{2})}{(x-{(2n-2)})}^{\frac{n}{2}}&\text{if} \  \  $n$ \ is \ even\\
(-x)(n-x)^{n-2}((2n-1)-x)^{n}& \text{if} \  \ $n$ \ is \ odd.
\end{array}\right.
$$
This completes the proof.
\end{proof}
\begin{cor}
$$
LE(\Gamma _{D_{2n}})=\left\{
\begin{array}{rl}
\frac{2n(n^{2}-4n+6)}{2n-2}\ \text{if} \ \ $n$\  is \ even \\
3n(n-1) \ \text{if} \ \ $n$ \ is \ odd.
\end{array}\right.
$$
\end{cor}
 
In Table $2$, the Laplacian characteristic polynomial and eigenvalues of non-commuting graphs of some dihedral groups is given.\\ 
\begin {table}[h]
\begin{center} 
\begin {tabular}{|l|l|l|l|} 
\hline
groups & characteristic polynomial & eigenvalues & energy\\
\hline
$D_{8}$ & $P_{A}(x)=x(x-4)^{3}(x-6)^{2}$ & $(0)_{1},(4)_{3},(6)_{2}$ & 8\\
$D_{10}$ & $P_{A}(x)=(-x)(x-5)^{3}(9-x)^{5}$ & $(0)_{1},(5)_{3},(9)_{5}$ & 60\\
$D_{12}$ & $P_{A}(x)=x(x-6)^{3}(x-8)_{3}(x-10)_{3}$ & $(0)_{1},(6)_{3},(8)_{3},(10)_{3}$ & $\frac{108}{5}$\\ 
$D_{14}$ & $P_{A}(x)=(-x)(x-7)^{5}(13-x)^{7}$ & $(0)_{1},(7)_{5},(13)_{7}$ & 126\\
$D_{16}$ & $P_{A}(x)=x(x-8)^{5}(x-12)^{4}(x-14)^{4}$& $(0)_{1},(8)_{5}, (12)_{4}, (14)_{4}$ & $\frac{304}{7}$ \\
\hline
\end{tabular} 
\end{center} 
\caption{The characteristic polynomial, eigenvalues and Laplacian energy of some graphs.}
\end{table} 
\section{The energy of non-commuting graph of direct product of groups}
In this section, we calculate the energy of non-commuting graph of $D_{2n}\times D_{2n}$ and $G\times H$, which G is a non-abelian finite group and H is an abelian finite group. 
\begin{thm} 
Let $G$ be a non-abelian  finite group and $H$ be an abelian group of order $n$. Then 
$E(\Gamma_{G\times H})=nE(\Gamma_{G}).$
\end{thm}
\begin{proof}
 Let $h_{1},h_{2},\ldots,h_{n}$ be elements of $H$. Suppose $A(G)$ and $A(H)$ be adjacency matrices of non-commuting graph of groups $G$ and $H$, respectively. Then the adjacency matrix of $\Gamma_{G\times H}$ as the following form \\
$$
A(\Gamma_{G\times H})=\begin{bmatrix} 
A(G) & A(G) & \ldots & A(G) \\ 
A(G) & A(G) & \ldots & A(G) \\ 
\vdots & \vdots & \ddots & \vdots \\ 
A(G) & A(G) & \ldots & A(G) \\ 
\end{bmatrix}.
$$
We have\\
$$
det(A(\Gamma_{G\times H})-Ix)=\begin{vmatrix} 
A(G)-Ix & A(G) & \ldots & A(G) \\ 
A(G) & A(G)-Ix & \ldots & A(G) \\ 
\vdots & \vdots & \ddots & \vdots \\ 
A(G) & A(G) & \ldots & A(G)-Ix \\ 
\end{vmatrix}.
$$
But 
\begin{align*}
det(A(\Gamma_{G\times H})-Ix)&=\left|(A(G)-Ix)+(n-1)A(G)\right| \left|(A(G)-Ix)-A(G)\right|^{n-1}\\
&=\left| nA(G)-Ix\right| \left|-Ix\right| ^{n-1}.
\end{align*}
Thus
 $$P_{A(\Gamma_{G\times H})}(x) =n\left| -Ix\right| ^{n-1} P_{A(G)}(\frac{x}{n}).$$
If $\lambda_{1},\lambda_{2},\ldots,\lambda_{t}$ be eigenvalues of $A(G)$, then 
$$x_{1}=n\lambda_{1},\  x_{2}=n\lambda_{2}, \ \ldots ,\  x_{t}=n\lambda_{t}\  ,$$
are eigenvalues of $A(\Gamma_{G\times H})$. Therefore
\begin{align*}
E(\Gamma_{G\times H}) &=\sum_{i=1}^{t}\left|{\lambda_{i}}\right| \\
&=\left|{n\lambda_{1}}\right|+\left|{n\lambda_{2}}\right|+\ldots+\left|{n\lambda_{t}}\right|\\
&={n{(\left|{\lambda_{1}}\right|+\ldots+\left|{\lambda_{t}}\right|)}}\\
&=nE(\Gamma_{G}).
\end{align*}
\end{proof}
\begin{lem}\label{7} 
Let $G=D_{8}$. Then
$$P_{\Gamma_{G\times G}}(x)=(-x)^{45}(-x+8)(-x-4)^{4}(x^{2}+8x-32)^{4}(x^{2}-40x-128)$$
and 
$$E(\Gamma_{G\times G})=\sum_{i=1}^{60}\left|{\lambda_{i}}\right| =8\big(3+\sqrt{33}+4\sqrt{3}\big).$$
\end{lem}
\begin{proof}
For group $G$, we consider the adjacency matrix $A^{0}(G)$ similar to $A(G)$ by adding vertices $Z(G)$ to $A(G)$. Adjacency matrices of $A(G)$ and $A^{0}(G)$ differ only in some zero rows and zero columns. We have 
$$ 
\small
A^{0}(\Gamma_{G\times G})=
\begin{bmatrix} 
A^{0}(G) & A^{0}(G) & A^{0}(G) & A^{0}(G) & A^{0}(G) & A^{0}(G) & A^{0}(G) & A^{0}(G) \\ 
A^{0}(G) & A^{0}(G) & A^{0}(G) & A^{0}(G) & A^{0}(G) & A^{0}(G) & A^{0}(G) & A^{0}(G) \\ 
A^{0}(G) & A^{0}(G) & A^{0}(G) & A^{0}(G) & J & J & J & J \\
A^{0}(G) & A^{0}(G) & A^{0}(G) & A^{0}(G) & J & J & J & J \\
A^{0}(G) & A^{0}(G) & J & J & A^{0}(G) & A^{0}(G) & J & J \\ 
A^{0}(G) & A^{0}(G) & J & J & A^{0}(G) & A^{0}(G) & J & J \\ 
A^{0}(G) & A^{0}(G) & J & J & J & J &A^{0}(G) & A^{0}(G) \\ 
A^{0}(G) & A^{0}(G) & J & J & J & J &A^{0}(G) & A^{0}(G) \\ 
\end{bmatrix}.
$$
By direct calculations
$$
det(A^{0}(\Gamma_{G\times G})-Ix)=\left| -Ix\right|^4 \left| 2A(G)^{0}-Ix-2J\right|^{2}
\small 
\begin{vmatrix}
2A(G)^{0}-Ix & 6A(G)^{0} \\
2A(G)^{0} & 2A(G)^{0}-Ix+4J\\
\end{vmatrix}.
$$
Thus\\
$\left| 2A(G)^{0}-Ix-2J\right|^2 =$
$$
\small
\begin{vmatrix}
-x-2 & -2 & 0 & 0 & 0 & 0 & -2 & -2 \\
-2 & -x-2 & 0 & 0 & 0 & 0 & -2 & -2 \\
0 & 0 & -x-2 & -2 & 0 & 0 & -2 & -2 \\
0 & 0 & -2 & -x-2 & 0 & 0 & -2 & -2 \\
0 & 0 & 0 & 0 & -x-2 & -2 & -2 & -2 \\
0 & 0 & 0 & 0 & -2 & -x-2 & -2 & -2 \\
-2 & -2 & -2 & -2 & -2 & -2 & -x-2 & -2 \\
-2 & -2 & -2 & -2 & -2 & -2 & -2 & -x-2 \\
\end{vmatrix}^2
$$
$\hspace{3.2cm}=(-x)^{8}(-x-4)^{4}(x^{2}+8x-32)^{2}.$\\
Now assume that\\
\\
$$
\begin{vmatrix}
2A(G)^{0}-Ix & 6A(G)^{0} \\
2A(G)^{0} & 2A(G)^{0}-Ix+4J\\
\end{vmatrix}
=M.
$$
By calculating
$$
detM=(-x)^{9}(x^2+8x-32)^2(x^2-40x-128)(-x+8).
$$
According to the description of the beginning of the proof, we conclude that
\begin{align*}
P_{\Gamma_{G\times G}}(x) &=\frac{(-x)^{49}(-x+8)(-x-4)^{4}(x^{2}+8x-32)^{4}(x^{2}-40x-128)}{(-x)^{4}}\\
\\
&=(-x)^{45}(-x+8)(-x-4)^{4}(x^{2}+8x-32)^{4}(x^{2}-40x-128)
\end{align*}
and 
$$E(\Gamma_{G\times G})=\sum_{i=1}^{60}\left|{\lambda_{i}}\right| =8\Big(3+\sqrt{33}+4\sqrt{3}\Big).$$
\end{proof} 
\begin{thm}
Let $G=D_{2n}$. If $n$ is even and $n>4$, then 
\begin{align*}
P_{\Gamma _{G\times G}}(x)&=\frac{(-x)^{\frac{15{n^2}}{4}-2n-3}}{(-x)^{4}}(-x-4)^{{\frac{n^2}{4}}-n+1}\\
&\big(-x^{3}-(2n+4)x^{2}+16n(n-2)\big)^{n-2}f(x),
\end{align*}
where $f(x)$ is a  polynomial of degree $8$.  Also the spectrum of $\Gamma _{G\times G}$ is equal to\\
$spec(\Gamma _{G\times G})=$
$$
\tiny
\left(
\begin{array}{cccccccccccccc}
-4 & 0 &\alpha_{1} &\alpha_{2}&\alpha_{3}&\gamma_{1}&\gamma_{2}& \gamma_{3} &\gamma_{4}&\gamma_{5}&\gamma_{6}&\gamma_{7}&\gamma_{8}\\
{\frac{n^2}{4}}-n+1&\frac{15{n^2}}{4}-2n-7 &{n-2} & {n-2}&{n-2} 
&1&1&1&1&1&1&1&1
\end{array}
\right),
$$
where  $\gamma_{1}$, $\gamma_{2}$, $ \gamma_{3}$,  $\gamma_{4}$, $\gamma_{5}$, $\gamma_{6}$, $\gamma_{7}$, $\gamma_{8}$ and $\alpha_{1}$, $\alpha_{2}$, $\alpha_{3}$  are roots  of $f(x)$  and  $\big(-x^{3}-(2n+4)x^{2}+16n(n-2)\big)$, respectively. 
\end{thm}
\begin{proof}
By the proof of Lemma \ref{7}, we have
$$A^{0}_{ij}(\Gamma _{G\times G})=\left\{
\begin{array}{rl}
A^{0}(G) & \ \ 1\leq i\leq 2n , 1\leq j\leq 2 \\
A^{0}(G) & \ \ 1\leq i\leq 2 , 1\leq j\leq 2n \\
A^{0}(G) & \ \ 3\leq i , j\leq n  \\
A^{0}(G) & \ \  i=k+t , j=k+s ; t, s=0 \ or \ 1 \\
 & \ \ and  \ k= n+1, \ldots , 2n-1 \\
J &  \ \ o.w
\end{array}\right.
$$
 By direct calculations\\
\begin{align*}
P_{\Gamma _{G\times G}}(x)&=\frac{|-Ix|^{\frac{3n}{2}-2}}{(-x)^{4}} |2A^{0}(G)-Ix-2J|^{\frac{n}{2}-1}\\
&\begin{vmatrix}
2A^{0}(G)-Ix & (n-2)A^{0}(G) & nA^{0}(G) \\
2A^{0}(G) & (n-2)A^{0}(G)-Ix & nJ \\
2A^{0}(G) & (n-2)J & 2A^{0}(G)-Ix+(n-2)J \\
\end{vmatrix}. 
\end{align*}
Thus
\begin{align*}
\left|2A^{0}(G)-Ix-2J\right|^{\frac{n}{2}-1}&=\Big((-x)^{\frac{3n}{2}-2} (-x-4)^{\frac{n}{2}-1}\Big)^{\frac{n}{2}-1}\\
& \Big(-x^{3}-(2n+4)x^{2}+16n(n-2)\Big)^{\frac{n}{2}-1}.
\end{align*}
Now assume that 
$$
\small 
\begin{bmatrix}
2A^{0}(G)-Ix & (n-2)A^{0}(G) & nA^{0}(G) \\
2A^{0}(G) & (n-2)A^{0}(G)-Ix & nJ \\
2A^{0}(G) & (n-2)J & 2A^{0}(G)-Ix+(n-2)J \\
\end{bmatrix}
=M.
$$
By direct calculations, we have
$$det(M)=(-x)^{\frac{9n}{2} -5}$$
$$ 
\tiny
\left|
\begin{array}{cccc|ccccc|ccccc}
-x&4&\ldots&4& 0&b&\ldots &b&0&0&2n&\ldots &2n&0\\
b&-x&\ldots&4&a&0&\ldots&b&0&c&0&\ldots&2n&0 \\
\vdots & \vdots &\ddots & \vdots &\vdots & \vdots &\ddots & \vdots&\vdots &\vdots & \vdots &\ddots & \vdots&\vdots\\
b&4&\ldots&-x&a&b&\ldots& 0&0&c&2n&\ldots&0&0\\
\hline
0&4&\ldots&4 &-x&b&\ldots&b&0 &c&2n&\ldots&2n&2n\\
b&0&\ldots&4 &a&-x&\ldots&b&0 &c&2n&\ldots&2n&2n\\
\vdots & \vdots &\ddots & \vdots &\vdots & \vdots &\ddots & \vdots&\vdots &\vdots & \vdots &\ddots & \vdots&\vdots\\
b&4&\ldots&0 &a&b&\ldots&-x&0& c&2n&\ldots&2n&2n\\
0&0&\ldots&0 &0&0&\ldots&0&-x &c&2n&\ldots&2n&2n\\
\hline
0&4&\ldots&4 &a&b&\ldots&b&b&-x+a&2n&\ldots&2n&b\\
b&0&\ldots&4 &a&b&\ldots&b&b&c&-x+b&\ldots&2n&b\\
\vdots & \vdots &\ddots & \vdots &\vdots & \vdots &\ddots & \vdots&\vdots &\vdots & \vdots &\ddots & \vdots&\vdots\\
b&4&\ldots&0 &a&b&\ldots&b&b&c&2n&\dots&-x+b&b\\
0&0&\ldots&0 &a&b&\ldots&b&b&a&b&\ldots&b&-x+b
\end{array}
\right|,
$$
where $a= (n-2)^{2}$, $b=2(n-2)$ and $c=n(n-2)$. Therefore
$$det(M)={\frac{1}{2}}(-x)^{\frac{9n}{2}-5}(-2x)^{\frac{n}{2}+2}$$
$$ 
\tiny
\left|
\begin{array}{cccc|cccc|c}
\frac{-x}{2}&4&\ldots&4 & \frac{-x+a}{2} & 2n&\ldots & 2n& \frac{b}{2}\\
b&\frac{-x}{2}&\ldots&4 & c &\frac{-x+b}{2}&\ldots& 2n& \frac{b}{2} \\
\vdots & \vdots &\ddots & \vdots &\vdots &\vdots& \ddots & \vdots&\vdots \\
b&4&\ldots&\frac{-x}{2}&c&2n&\ldots&\frac{-x+b}{2}&\frac{b}{2}\\
\hline
x+\frac{c}{2}&0&\ldots&0&\frac{cx+ca+4na}{2x}&\frac{4c}{x}&\ldots&\frac{4c}{x}&\frac{na+4c}{x}\\
0&x+n&\ldots &0&\frac{2na}{x}&\frac{nx+6c}{x}&\ldots&\frac{4c}{x}&\frac{6c}{x}\\
\vdots & \vdots &\ddots & \vdots &\vdots &\vdots& \ddots & \vdots&\vdots \\
0&0&\ldots&x+n&\frac{2na}{x}&\frac{4c}{x}&\ldots &\frac{nx+6c}{x}&\frac{6c}{x}\\
\hline
\frac{c}{2}&n&\ldots&n&\frac{(c+2a)x+ca+4na}{2x}&\frac{(3n-4)x+6c}{x}&\ldots&\frac{(3n-4)x+6c}{x}& -x+{\frac{bx+na+nc+4c}{x}}\\
\end{array}
\right|.
$$
Let
$$ 
A=
\tiny
\left[
\begin{array}{cccc|cccc}
\frac{-x}{2}& 4& \ldots & 4 & \frac{-x+a}{2}& 2n& \ldots & 2n\\
b&\frac{-x}{2}&\ldots&4&c&\frac{-x+b}{2}&\ldots&2n\\
\vdots & \vdots &\ddots & \vdots &\vdots &\vdots& \ddots & \vdots\\
b&4&\ldots&\frac{-x}{2}&c&2n&\ldots&\frac{-x+b}{2}\\
\hline
x+\frac{c}{2}&0&\ldots&0&\frac{cx+ca+4na}{2x}&\frac{4c}{x}&\ldots&\frac{4c}{x}\\
0&x+n&\ldots &0&\frac{2na}{x}&\frac{nx+6c}{x}&\ldots&\frac{4c}{x}\\
\vdots & \vdots &\ddots & \vdots &\vdots &\vdots& \ddots & \vdots \\
0&0&\ldots&x+n&\frac{2na}{x}&\frac{4c}{x}&\ldots &\frac{nx+6c}{x}
\end{array}
\right],
$$
$$
B^{T}=
\begin{bmatrix}
\frac{b}{2}&\frac{b}{2}&\cdots&\frac{b}{2}&\frac{na+4c}{x}&\frac{6c}{x}&\cdots&\frac{6c}{x}
\end{bmatrix},
$$
$$
C=
\begin{bmatrix}
\frac{c}{2}&n&\ldots&n&\frac{(c+2a)x+ca+4na}{2x}&\frac{(3n-4)x+6c}{x}&\ldots&\frac{(3n-4)x+6c}{x}
\end{bmatrix}
$$
and
$$d= -x+{\frac{bx+na+nc+4c}{x}},$$
where $B^{T}$ is transpose of $B$. It follows that
\begin{align*}
det(M)&={\frac{1}{2}}(-x)^{\frac{9n}{2}-5}(-2x)^{\frac{n}{2}+2}\big[ (d-1)det(A)+det(A-BC)\big]\\
&=(-x)^{\frac{9n}{2}-5}\Big(-x^{3}-(2n+4)x^{2}+16n(n-2) \Big)^{\frac{n}{2}-1} f(x),
\end{align*}
where $f(x)$ is a  polynomial of degree $8$. 
This implies that 
\begin{align*}
P_{\Gamma _{G\times G}}(x)&=\frac{(-x)^{\frac{15{n^2}}{4}-2n-3}}{(-x)^{4}}(-x-4)^{{\frac{n^2}{4}}-n+1}\\
&\big(-x^{3}-(2n+4)x^{2}+16n(n-2)\big)^{n-2}f(x).
\end{align*}
Now let $\gamma_{1}$, $\gamma_{2}$, $ \gamma_{3}$,  $\gamma_{4}$, $\gamma_{5}$, $\gamma_{6}$, $\gamma_{7}$ and $\gamma_{8}$ be roots of $f(x)$ in $P_{\Gamma _{G\times G}}(x)$. 
Suppose $\alpha_{1}$, $\alpha_{2}$ and $\alpha_{3}$  are  roots of\\
$ \big(-x^{3}-(2n+4)x^{2}+16n(n-2)\big)$.
Then\\
$spec(\Gamma _{G\times G})=$
$$
\tiny
\left(
\begin{array}{cccccccccccccc}
-4 & 0 &\alpha_{1} &\alpha_{2}&\alpha_{3}&\gamma_{1}&\gamma_{2}& \gamma_{3} &\gamma_{4}&\gamma_{5}&\gamma_{6}&\gamma_{7}&\gamma_{8}\\
{\frac{n^2}{4}}-n+1&\frac{15{n^2}}{4}-2n-7 &{n-2} & {n-2}&{n-2} 
&1&1&1&1&1&1&1&1
\end{array}
\right).
$$
\end{proof}
\begin{cor}
\begin{align*}
\mathcal{E}(\Gamma_{D_{2n}\times D_{2n}})=&
(n^{2}-4n+4)+(n-2)\Big(|\alpha_{1}| +|\alpha_{2}| +|\alpha_{3}|\Big)\\
&+|\gamma_{1}| +|\gamma_{2}| + |\gamma_{3}|+|\gamma_{4}| +|\gamma_{5}| + |\gamma_{6}|+ |\gamma_{7}|+ |\gamma_{8}|.
\end{align*}
\end{cor}
\begin{exam}
\begin{align*}
P_{\Gamma_{D_{12}\times D_{12}}}(x)&=(-x)^{116}(x^{3}+16x^{2}-384)^{4}(x+4)^{4}(x+24)(x^{3}-384x+2304)\\
&(x^{4}-104x^{3}-1152x^{2}+5376x+55296)
\end{align*}
and
\begin{align*}
\mathcal{E}(\Gamma_{D_{12}\times D_{12}})=&
40+4\big(|\alpha_{1}| +|\alpha_{2}| +|\alpha_{3}|\big)+\big(|\beta_{1}| +|\beta_{2}| +|\beta_{3}|\big)\\
&+\big(|\gamma_{1}| +|\gamma_{2}| +|\gamma_{3}|+|\gamma_{4}|\big).\\
\end{align*}
where $\alpha_{1}$, $\alpha_{2}$, $\alpha_{3}$ and  $\beta_{1}$, $\beta_{2}$, $\beta_{3}$  are roots of\\
$(x^{3}+16x^{2}-384)$ and $(x^{3}-384x+2304)$, respectively. Also $\gamma_{1}$, $\gamma_{2}$, $\gamma_{3}$, $\gamma_{4}$ are roots of  $(x^{4}-104x^{3}-1152x^{2}+5376x+55296)$.
\end{exam}
\begin{thm}
Let $G$ be a finite non-abelian group, where $\left| Z(G)\right|=t$. Then \\
\begin{align*}
P_{\Gamma_{G\times S_{3}}}(x)&=\frac{\left|-Ix\right|}{(-x)^{t}} \left|A^{0}(G)-Ix-J\right|^{2} \\
&\begin{vmatrix} 
A^{0}(G)-Ix & 2A^{0}(G) & 3A^{0}(G) \\ 
A^{0}(G) & 2A^{0}(G)-Ix & 3J \\
A^{0}(G) & 2J & A^{0}(G)-Ix+2J \\ 
\end{vmatrix}.
\end{align*}
\end{thm}
\begin{proof} 
By the proof of Lemma \ref{7}, the adjacency matrix of $\Gamma_{G\times S_{3}}$ as following \\
$$A^{0}_{ij}(\Gamma _{G\times S_{3}})=\left\{
\begin{array}{rl}
A^{0}(G) & \ \  i=1 , \ 1\leq j\leq 6 \\
A^{0}(G) & \ \ 1\leq i\leq 6 , \  j=1 \\
A^{0}(G) & \ \ 2\leq i , j\leq 3 \\
A^{0}(G) & \ \  i=j=4,\  5,\  6\\
J &  \ \ o.w
\end{array}\right.
$$
Similar to the proof of Lemma \ref{7}, we have\\
\begin{align*}
P_{\Gamma_{G\times S_{3}}}(x)&=\frac{\left|-Ix\right|}{(-x)^{t}} \left|A^{0}(G)-Ix-J\right|^{2}\\
&\begin{vmatrix} 
A^{0}(G)-Ix & 2A^{0}(G) & 3A^{0}(G) \\ 
A^{0}(G) & 2A^{0}(G)-Ix & 3J \\
A^{0}(G) & 2J & A^{0}(G)-Ix+2J \\ 
\end{vmatrix}.
\end{align*}
\end{proof}
\begin{thm} 
Let $G$ be a  finite non-abelian group, where $\left| Z(G)\right|=t$. If $n$ is even, then 
\begin{align*}
P_{\Gamma _{G\times D_{2n}}}(x)&=\frac{\left|-Ix\right|^{\frac{3n}{2}-2}}{(-x)^{2t}} \left|2A^{0}(G)-Ix-2J\right|^{\frac{n}{2}-1}\\
&\begin{vmatrix}
2A^{0}(G)-Ix & (n-2)A^{0}(G) & nA^{0}(G) \\
2A^{0}(G) & (n-2)A^{0}(G)-Ix & nJ \\
2A^{0}(G) & (n-2)J & 2A^{0}(G)-Ix+(n-2)J \\
\end{vmatrix}.
\end{align*}
\end{thm}
\begin{proof}
By the proof of Lemma \ref{7}, we have
$$A^{0}_{ij}(\Gamma _{G\times D_{2n}})=\left\{
\begin{array}{rl}
A^{0}(G) & \ \ 1\leq i\leq 2n , 1\leq j\leq 2 \\
A^{0}(G) & \ \ 1\leq i\leq 2 , 1\leq j\leq 2n \\
A^{0}(G) & \ \ 3\leq i , j\leq n  \\
A^{0}(G) & \ \  i=k+t , j=k+s ; t, s=0 \ or \ 1 \\
 & \ \ and  \ k= n+1, \ldots , 2n-1 \\
J &  \ \ o.w
\end{array}\right.
$$
By direct calculations\\
\begin{align*}
P_{\Gamma _{G\times D_{2n}}}(x)&=\frac{\left|-Ix\right|^{\frac{3n}{2}-2}}{(-x)^{2t}} \left|2A^{0}(G)-Ix-2J\right|^{\frac{n}{2}-1}\\
&\begin{vmatrix}
2A^{0}(G)-Ix & (n-2)A^{0}(G) & nA^{0}(G) \\
2A^{0}(G) & (n-2)A^{0}(G)-Ix & nJ \\
2A^{0}(G) & (n-2)J & 2A^{0}(G)-Ix+(n-2)J \\
\end{vmatrix}. 
\end{align*}
\end{proof}
\begin{thm} 
Let $G$ be a finite non-abelian group, where $\left| Z(G)\right|=t$. Then
\begin{align*}
P_{\Gamma_{G\times {D_{8}}}}(x)&=\frac{\left| -Ix\right|^4}{(-x)^{2t}} \left| 2A^{0}(G)-Ix-2J\right|^{2}\\ 
&\begin{vmatrix}
2A^{0}(G)-Ix & 6A^{0}(G) \\
2A^{0}(G) & 2A^{0}(G)-Ix+4J\\
\end{vmatrix}.
\end{align*}
\end{thm} 
\begin{proof}
It is similar to the proof of Lemma \ref{7}.
\end{proof}





\bibliographystyle{amsplain} 

\bibliographystyle{amsplain}
\bibliography{xbib}

\begin{thebibliography}{5}
\bibitem{Ha} A. Abdollahi, S. Akbari and H. R. Maimani, Non-commuting graph of a group, J.Algebra 298(2006),no .2,468-492. 
\bibitem{Ha} D. Cvetkovi\'c, M. Doob, H. Sachs, Spectra of graphs -Theory and application, Academic press, Newyork, 1980.
\bibitem{Ha} I. Gutman, O. E. Polansky, Mathematical consepts in organic chemistry, springer, Berlin, 1986.
\bibitem{Ha} I.Gutman, The energy of a graph: old and new results, in: A. Betten, A. Kohnert, R. Laue and A. Wassermann(Eds)-verlag, Berlin, 2001, pp 196-211.
\bibitem{Ha} D. Stevanovic, I. Gutman, M. U. Rehman, On spectral radius energy of complete multipartite graphs, received 18 june 2013, accepted 8 march 2014 , published online 8 august 2014.
\bibitem{Ha} C. Delorme, Eigenvalues of complete multipartite graphs, Discrete math.312(2012), 2532-2535.
\bibitem{Ha} B. H. Neumann, A problem of Paul Erd\"os on groups, J. Aust. Math. Soc. Ser. A 21(1976)467-472. 
\end{thebibliography}
\end{document}